\documentclass[12pt]{article}
\usepackage{amsthm, amsmath, amssymb, amsfonts}
\usepackage{tikz}
\usepackage{microtype}
\usepackage{url}

\newtheorem{thm}{Theorem}[section]

\newtheorem{lemma}[thm]{Lemma}
\newtheorem{prop}[thm]{Proposition}

\newtheorem{conj}[thm]{Conjecture}

\newcommand{\beq}[1]{\begin{equation}\label{#1}}
\newcommand{\enq}[0]{\end{equation}}

\tikzstyle{V}=[draw,circle, fill=black, minimum size=4pt,inner sep=0pt]

\begin{document}

\renewcommand{\thefootnote}{\fnsymbol{footnote}}

\title{The number of independent sets in a graph with small maximum degree}

\author{
David Galvin\thanks{Department of Mathematics,
University of Notre Dame, 255 Hurley Hall, Notre Dame IN
46556, USA; {\tt dgalvin1@nd.edu}}
\and
Yufei Zhao\thanks{Department of Mathematics, Massachusetts Institute of Technology,
Cambridge, MA 02139, USA; {\tt yufeiz@mit.edu}}
}

\date{\today}

\maketitle

\begin{abstract}
Let ${\rm ind}(G)$ be the number of independent sets in a graph $G$. We show that if $G$ has maximum degree at most $5$ then
$$
{\rm ind}(G) \leq 2^{{\rm iso}(G)} \prod_{uv \in E(G)} {\rm ind}(K_{d(u),d(v)})^{\frac{1}{d(u)d(v)}}
$$
(where $d(\cdot)$ is vertex degree, ${\rm iso}(G)$ is the number of isolated vertices in $G$ and $K_{a,b}$ is the complete bipartite graph with $a$ vertices in one partition class and $b$ in the other), with equality if and only if each connected component of $G$ is either a complete bipartite graph or a single vertex. This bound (for all $G$) was conjectured by Kahn.

A corollary of our result is that if $G$ is $d$-regular with $1 \leq d \leq 5$ then
$$
{\rm ind}(G) \leq \left(2^{d+1}-1\right)^\frac{|V(G)|}{2d},
$$
with equality if and only if $G$ is a disjoint union of $V(G)/2d$ copies of $K_{d,d}$. This bound (for all $d$) was conjectured by Alon and Kahn and recently proved for all $d$ by the second author, without the characterization of the extreme cases.

Our proof involves a reduction to a finite search. For graphs with maximum degree at most $3$ the search could be done by hand, but for the case of maximum degree $4$ or $5$, a computer is needed.
\end{abstract}

\section{Introduction and statement of the result}

For a graph $G$ let ${\rm ind}(G)$ denote the number of independent sets of $G$ (sets of vertices no two of which are adjacent). Kahn \cite{Kahn} made the following conjecture concerning ${\rm ind}(G)$. (All graphs in this note are finite, undirected and simple.)
\begin{conj} \label{strong-ind_conj}
For any graph $G$,
$$
{\rm ind}(G) \leq 2^{{\rm iso}(G)} \prod_{uv \in E(G)} {\rm ind}(K_{d(u),d(v)})^{\frac{1}{d(u)d(v)}},
$$
where $d(x)$ is the degree of $x$, ${\rm iso}(G)$ is
the number of isolated vertices in $G$ and $K_{a,b}$ is the complete bipartite graph with $a$ vertices in one
partition class and $b$ in the other. Equivalently,
\begin{equation} \label{kahn-bound}
{\rm ind}(G) \leq
2^{{\rm iso}(G)}\prod_{uv \in E(G)}
\left(2^{d(u)}+2^{d(v)}-1\right)^\frac{1}{d(u)d(v)}
\end{equation}
\end{conj}
(Here we adopt the convention that the graph without vertices has $1$ independent set, and that a product over an empty set is $1$.) Note that there is equality in (\ref{kahn-bound}) when $G$ is the disjoint union of complete bipartite graphs together with some isolated vertices.

The special case of Conjecture \ref{strong-ind_conj} when $G$ is $d$-regular was made implicitly by Alon \cite{Alon} and appears explicitly in \cite{Kahn}. It was recently resolved by Zhao \cite{Zhao}.
\begin{thm} \label{reg-ind_thm}
For a $d$-regular graph $G$ with $d \geq 1$,
\begin{equation} \label{zhao-bound}
{\rm ind}(G) \leq \left(2^{d+1}-1\right)^{\frac{|V(G)|}{2d}}.
\end{equation}
\end{thm}
Kahn \cite{Kahn} proved (\ref{zhao-bound}) for bipartite $G$ using entropy methods, and Zhao's proof also uses entropy in the sense that he deduces the non-bipartite result from the bipartite result.

\medskip

The aim of this note is to prove the following.
\begin{thm} \label{thm-non_bip_delta_5}
For all $G$ with maximum degree at most $5$, (\ref{kahn-bound}) holds. There is equality in (\ref{kahn-bound}) if and only if each connected component of $G$ is either a complete bipartite graph or a single vertex.
\end{thm}
As a corollary we obtain a new proof of Zhao's result in the case $d \leq 5$ that does not use entropy and that in addition provides the information that there is equality in (\ref{zhao-bound}) if and only if $G$ is a disjoint union of $V(G)/2d$ copies of $K_{d,d}$.

\medskip

Our proof is motivated by the observation that for $G$ with at least one vertex the quantity ${\rm ind}(G)$ satisfies the recursion
\begin{equation} \label{ind-recurrsion}
{\rm ind}(G) = {\rm ind}(G-x) + {\rm ind}(G-x-N(x))
\end{equation}
where $x$ is any vertex of $G$, and $N(x)$ denotes its set of neighbours. To see (\ref{ind-recurrsion}), consider first those independent sets which do not include $x$, and then those that do.

Set
$$
\Pi(G) = 2^{{\rm iso}(G)}\prod_{uv \in E(G)}
\left(2^{d(u)}+2^{d(v)}-1\right)^\frac{1}{d(u)d(v)}.
$$
So Conjecture \ref{strong-ind_conj} asserts that for all $G$, ${\rm ind}(G) \leq \Pi(G)$. Say that $x \in V(G)$ is {\em good} for $G$ if
\begin{equation} \label{def-good}
\Pi(G) \geq \Pi(G-x) + \Pi(G-x-N(x)).
\end{equation}
Note that (\ref{def-good}) is satisfied with equality if the connected component of $G$ that contains $x$ is either $K_{a,b}$ or a single vertex. Note also that if $uv \in E(G)$ is such that both $u$ and $v$ are at distance at least $3$ from $x$, it contributes the same multiplicative factor to both the right- and left-hand sides of (\ref{def-good}). This observation reduces the verification of (\ref{kahn-bound}) for graphs with bounded degree to a finite search.

\medskip

Here is a conjecture that implies Conjecture \ref{strong-ind_conj} (we will prove the implication in the next section).
\begin{conj} \label{stronger-ind_conj}
For every graph $G$ with at least one vertex there is $x \in V(G)$ that is good for $G$.
\end{conj}
Note that it is not true that for every $G$ {\em every} $x \in V(G)$ is good for $G$. See Figure \ref{fig:not-good} for an example.

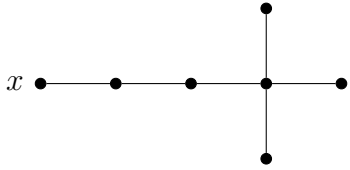
\begin{figure}[ht!] \center
\begin{tikzpicture}
	\node[V, label=left:$x$] (0) {};
	\node[right of=0, V] (1) {};
	\node[right of=1, V] (2) {};
	\node[right of=2, V] (3) {};
	\node[right of=3, V] (4) {};
	\node[above of=3, V] (5) {};
	\node[below of=3, V] (6) {};
	\draw (0)--(1)--(2)--(3) edge (4) edge (5) edge (6);
\end{tikzpicture}
\caption{An example where vertex $x$ is not good.\label{fig:not-good}}
\end{figure}

We also propose a strengthening of Conjecture \ref{strong-ind_conj} that is suggested by our computations.

\begin{conj} \label{stronger-ind_conj-Delta}
Let $x \in V(G)$ be a vertex of maximum degree. Then $x$ is good for $G$.
\end{conj}

We will prove Theorem \ref{thm-non_bip_delta_5} by verifying Conjecture \ref{stronger-ind_conj} for all bipartite $G$ with maximum degree at most $5$, and then using Zhao's observation to deduce \eqref{kahn-bound} for non-bipartite $G$ with maximum degree at most $5$. Here is the main result of this note.

\begin{prop} \label{prop-Delta-at-most-5}
Let $G$ be a bipartite graph with at least one vertex.
\begin{enumerate}
\item \label{main-thm_1}
If $G$ has maximum degree at most $5$, and $v$ is a vertex of minimum degree, then there exists a good vertex $x$ which is either $v$ or one of the neighbors of $v$.
\item \label{main-thm_2}
If $G$ has maximum degree at most $4$ then any vertex $x$ of maximum degree is good.
\item \label{main-thm_3}
If $G$ is $d$-regular then any vertex $x$ is good.
\end{enumerate}
Furthermore, in each case, equality holds in \eqref{def-good} if and only if the connected component of $x$ is either an isolated vertex or a complete bipartite graph.
\end{prop}

The process of verifying that every graph with maximum degree at most $\Delta$ has a good vertex is a finite one (with the number of cases depending on $\Delta$) and so in principle we could use the method outlined here to verify Conjecture \ref{strong-ind_conj} for the set of graphs with maximum degree at most $\Delta$ for some larger values of $\Delta$. However, the number of cases to be examined grows quickly with $\Delta$, and even the case $\Delta=6$ seems out of reach at the moment.

In Section \ref{sec-reduction} we show that Proposition \ref{prop-Delta-at-most-5} implies Theorem \ref{thm-non_bip_delta_5}. In Section \ref{sec-proof} we prove Proposition \ref{prop-Delta-at-most-5}.

\section{Reduction to Proposition \ref{prop-Delta-at-most-5}} \label{sec-reduction}

The following lemma verifies that Conjecture \ref{stronger-ind_conj} implies Conjecture \ref{strong-ind_conj}, thereby reducing the problem to showing the existence of good vertices.

\begin{lemma} \label{lem-imp}
Let ${\cal G}$ be a family of graphs that is closed under deleting vertices. If for every $G \in {\cal G}$ with at least one vertex there is $x \in V(G)$ that is good for $G$, then \eqref{kahn-bound} holds for all $G \in {\cal G}$.
\end{lemma}

\begin{proof}
We proceed by induction on the number of vertices of $G$. For $G$ with no vertices, there is nothing to prove. Otherwise, choose a good vertex $x$ for $G$. We have
\begin{align}
{\rm ind}(G) & =  {\rm ind}(G-x) + {\rm ind}(G-x-N(x)) \nonumber \\
 &\leq  \Pi(G-x) + \Pi(G-x-N(x)) \label{using_induction-1} \\
 &\leq  \Pi(G) \label{using-xgood}
\end{align}
where in \eqref{using_induction-1} we use the induction hypothesis and in \eqref{using-xgood} we use the fact that $x$ is good for $G$.
\end{proof}

The next two lemmas show that Proposition \ref{prop-Delta-at-most-5} implies Theorem \ref{thm-non_bip_delta_5}. First, we consider the case when the graph $G$ in Theorem \ref{thm-non_bip_delta_5} is bipartite. Then, we relax this restriction by using a recent result of Zhao \cite{Zhao}.

\begin{lemma} \label{lem-imp-bip}
Proposition \ref{prop-Delta-at-most-5} implies Theorem \ref{thm-non_bip_delta_5} for bipartite graphs $G$.
\end{lemma}

\begin{proof}
For a bipartite graphs $G$ with maximum degree at most $5$, Proposition \ref{prop-Delta-at-most-5} shows that $G$ contains a good vertex, so Lemma \ref{lem-imp} implies that \eqref{kahn-bound} holds. It remains to find the equality cases. Suppose that $G$ is a bipartite graph for which equality holds in \eqref{kahn-bound}. With $x$ chosen as in Proposition \ref{prop-Delta-at-most-5}, we have
\begin{align*}
\Pi(G) & =  {\rm ind}(G) \\
& =  {\rm ind}(G-x) + {\rm ind}(G-x-N(x)) \\
& \leq  \Pi(G-x) + \Pi(G-x-N(x)) \\
& \leq  \Pi(G)
\end{align*}
so that in fact $\Pi(G) = \Pi(G-x) + \Pi(G-x-N(x))$. Then the equality condition in Proposition \ref{prop-Delta-at-most-5} implies that the connected component of $x$ in $G$ is either an isolated vertex or a complete bipartite graph. Removing this component, the resulting graph $G'$ satisfies $\Pi(G') = {\rm ind}(G')$ so repeating the above argument we find that each connected component of $G$ is either a complete bipartite graph or a single vertex.
\end{proof}

\begin{lemma}
If Theorem \ref{thm-non_bip_delta_5} is true for bipartite graphs, then it is true in general.
\end{lemma}

\begin{proof}
Zhao \cite{Zhao} showed that
\begin{equation} \label{zhao-obsv}
{\rm ind}(G)^2 \leq {\rm ind}(G \times K_2)
\end{equation}
with equality if and only if $G$ is bipartite. Here $\times$ means tensor product (so $G \times K_2$ is the graph on vertex set $\{(v,i): V \in V(G), i \in \{0,1\}\}$ with $(u,i) \sim (v,j)$ if and only if $uv \in E(G)$ and $i \neq j$).

Now let $G$ be a non-bipartite graph with maximum degree at most $5$. Each edge $uv$ of $G$ with $d(u)=a$ and $d(v)=b$ gives rise to two edges $u'v''$ and $u''v'$ in $G \times K_2$ with $d(u')=d(u'')=a$ and $d(v')=d(v'')=b$, each isolated vertex of $G$ gives rise to two isolated vertices of $G \times K_2$, and $G \times K_2$ is a bipartite graph with maximum degree at most $5$. Combining these observations with \eqref{zhao-obsv} and Theorem \ref{thm-non_bip_delta_5} for bipartite graphs, we obtain
\begin{eqnarray*}
{\rm ind}(G)^2 & < & {\rm ind}(G \times K_2) \\
& \leq & 2^{{\rm iso}(G \times K_2)}\prod_{uv \in E(G \times K_2)} \left(2^{d(u)}+2^{d(v)}-1\right)^{\frac{1}{d(u)d(v)}} \\
& = & \left(2^{{\rm iso}(G)}\prod_{uv \in E(G)} \left(2^{d(u)}+2^{d(v)}-1\right)^{\frac{1}{d(u)d(v)}}\right)^2,
\end{eqnarray*}
as required. In particular, equality never occurs for non-bipartite graphs.
\end{proof}

Therefore, Theorem \ref{thm-non_bip_delta_5} reduces to Proposition \ref{prop-Delta-at-most-5}, which we prove in the next section.

\section{Proof of Proposition \ref{prop-Delta-at-most-5}} \label{sec-proof}

We begin with some notation. For $a, b > 0$ set
$$
f(a,b) = \left(2^a + 2^b -1\right)^{\frac{1}{ab}}.
$$
Note that $f(a,b)^{ab} = {\rm ind}(K_{a,b})$.
Here is a useful fact about the function $f(\cdot,\cdot)$. For all $0 < a' < a \leq 5$ and $0 < b' < b \leq 5$,
\begin{equation} \label{f(a,b)_fact_1}
\frac{f(a-a',b)f(a,b-b')}{f(a-a',b-b')f(a,b)} \geq 1.
\end{equation}
This relation in fact holds for all $a, b$, but in the sequel we only need it for $a, b \leq 5$, for which it is easily checked by hand.

\medskip

Fix a reference vertex $x \in V(G)$. Say that a vertex at distance $i$ from $x$ is a {\em level $i$} vertex (distance measured by number of vertices in a shortest path).
Say that an edge of $G$ is an {\em $ij$-edge} (with $i \leq j$) if it joins a level $i$ vertex and a level $j$ vertex (necessarily $j=i+1$; recall that $G$ is bipartite). Write $E_{ij}$ for the set of $ij$-edges. Adopt the convention of writing an $ij$-edge
as an ordered pair $(u,v)$ with the level of $u$ smaller than the level of $v$.

Write $G'$ for the component of $G$ that includes $x$, and $\Pi(G,x)$ for the contribution to $\Pi(G)$ from all other components. We have
$$
\Pi(G) = \Pi(G,x) 2^{{\rm iso}(G')}\prod_{ij \in \{01, 12, 23, 34, \ldots\}} \prod_{(u,v)
\in E_{ij}} f(d(u),d(v)).
$$
This is valid both when $G'$ contains more than one vertex (in which case ${\rm iso}(G')=0$) and when $G'=\{x\}$ (in which case ${\rm iso}(G')=1$). We also have
\begin{eqnarray*}
\Pi(G-x) & = & \Pi(G,x)2^{{\rm iso}(G'-x)}\prod_{(u,v) \in E_{12}} f(d(u)-1,d(v)) \times \\
& & ~~~~~~~~~~~~~~~~~~~~~~~~~~~~\prod_{ij \in \{23, 34, \ldots \}} \prod_{(u,v) \in E_{ij}}
f(d(u),d(v))
\end{eqnarray*}
and
\begin{eqnarray*}
\Pi(G-x-N(x)) & = & \Pi(G,x)2^{{\rm iso}(G'-x-N(x))} \times \\
& & ~~~~~~~~~~~~~ \prod_{(u,v) \in E_{23}} f(d(u)-d_{N(x)}(u),d(v)) \times \\
& & ~~~~~~~~~~~~~ \prod_{ij \in \{34, \ldots \}} \prod_{(u,v) \in E_{ij}}
f(d(u),d(v)),
\end{eqnarray*}
where $d_{N(x)}(u)$ denotes the number of neighbours of $u$ that are
contained in $N(x)$.

The condition that $x$ is good for $G$ is, in the notation we have just established,
\begin{align}
& 2^{{\rm iso}(G')}
\prod_{ij \in \{01,12,23\}} \prod_{(u,v) \in E_{ij}} f(d(u),d(v)) \nonumber \\
& \qquad \geq
2^{{\rm iso}(G'-x)}
\prod_{(u,v) \in E_{12}} f(d(u)-1,d(v))
 \prod_{(u,v) \in E_{23}} f(d(u),d(v)) \nonumber  \\
& \qquad \qquad +
2^{{\rm iso}(G'-x -N(x))}
\prod_{(u,v) \in E_{23}} f(d(u)-d_{N(x)}(u),d(v)).  \label{to_prove}
\end{align}
All terms involving $ij$-edges with $i \geq 3$ cancel out, as do all terms involving vertices at level $5$ or greater, or not in the same component as $x$.

\medskip

It is now clear that there is a finite process to verify \eqref{kahn-bound} for all $G \in {\cal G}^{\rm bipartite}(\Delta)$, the set of bipartite graphs with maximum degree at most $\Delta$, since there are only finitely many cases (the number depending on $\Delta$) for which \eqref{to_prove} needs to be checked. Indeed, we only need to check all possible configurations within an edge-radius of 4 about $x$. However, for $\Delta=5$, the number of cases is already too large to do an exhaustive search, so we need to make some observations that reduce the search.

\begin{lemma} \label{lem-obs_1}
Fix $G \in {\cal G}^{\rm bipartite}(\Delta)$ (with $\Delta \leq 5$) and $x \in V(G)$. Let $G'$ be obtained from $G$ by deleting all level $5$ and greater vertices and all vertices not in the same component as $x$, and adding edges (and level 4 vertices) to increase the degrees of all level $3$ vertices to $\Delta$ (without changing the degrees of level $0$, $1$ and $2$ vertices). If $x$ is good for $G'$ then it is good for $G$.
\end{lemma}

\begin{proof}
Let $v_1, \ldots, v_k$ be the level $3$ vertices of $G$ and let $a_j$ be the degree of $v_j$ for $1 \leq j \leq k$. If $a_j=\Delta$ for all $j$, then there is nothing to prove, as \eqref{to_prove} is then identical for both $G$ and $G'$. Otherwise, we have $a_j < \Delta$ for some $j$, without loss of generality $j=1$. We consider the graph $G''$ obtained from $G$ by deleting all level $5$ and greater vertices and all vertices not in the same component as $x$, and adding a single edge (and a single level 4 vertex) to increase the degree of vertex $v_1$ to $a_1+1$ (without changing the degrees of level $0$, $1$ and $2$ vertices). We claim that if $x$ is good for $G''$ then it is good for $G$; the lemma clearly follows from this by repeated application.

To see the claim, let the level 2 neighbours of $v_1$ have degrees $b_1, \ldots, b_\ell$ with the $j$th level $2$ neighbour having $b_j'$ level $1$ neighbours. By hypothesis we have $A' \geq B' + C'$ where $A'$ is the left-hand side of \eqref{to_prove} (applied to $G''$), $B'$ is the first term on the right-hand side and $C'$ is the second term. If $A, B$ and $C$ are the left-hand side and first and second terms of the right-hand side of \eqref{to_prove} (applied to $G$) then we have
$$
A'=A \prod_{i=1}^{\ell} \frac{f(b_i,a_1+1)}{f(b_i,a_i)},
$$
$$
B'=B \prod_{i=1}^{\ell} \frac{f(b_i,a_1+1)}{f(b_i,a_i)}
$$
and
$$
C'=C \prod_{i=1}^{\ell} \frac{f(b_i-b_i',a_1+1)}{f(b_i-b_i',a_1)}
$$
and so
$$
A \geq B + C \prod_{i = 1}^\ell \frac{f(b_i-b_i',a_1+1)f(b_i,a_i)}{f(b_i-b_i',a_1)f(b_i,a_1+1)}.
$$
So we get \eqref{to_prove} for $G$ from \eqref{f(a,b)_fact_1}.
\end{proof}

The next observation is a restatement of Proposition \ref{prop-Delta-at-most-5} (statement \ref{main-thm_3}).
\begin{lemma} \label{lem-regular}
If $G$ is $d$-regular and bipartite then any $x \in V(G)$ is good for $G$. There is equality in \eqref{to_prove} if and only if the component of $x$ is isomorphic to $K_{d,d}$.
\end{lemma}

\begin{proof}
Pick any $x \in V(G)$.
Each of the $d$ level $1$ vertices has $d-1$ level 2 neighbours, so there are in all $k$ level 2 vertices for some $d-1 \leq k \leq d(d-1)$. (The two extreme cases, $k=d-1$ and $k=d(d-1)$, correspond respectively to $G=K_{d,d}$ and to $G$ not having a 4-cycle through $x$). Label the level 2 vertices $v_1, \ldots, v_k$, and for each $v_i$ let $x_i=d-d_{N(x)}(v_i)$ (so $x_i$ is the number of level 3 neighbours that $v_i$ has). Note that $G$ has $d$ $01$-edges, $d(d-1)$ $12$-edges, and $kd$ edges leaving level 2 vertices, so $kd-d(d-1)$ $23$-edges. The left-hand side of \eqref{to_prove} is
$$
f(d,d)^{d+d(d-1)+kd-d(d-1)} = f(d,d)^{kd+d},
$$
the first term on the right-hand side is
$f(d-1,d)^{d(d-1)}f(d,d)^{kd-d(d-1)}$
and the second term is
$$
\prod_{i=1}^k f(x_i,d)^{x_i} = \prod_{i=1}^k \left(2^d + 2^{x_i} - 1\right)^\frac{1}{d}
$$
(where we have adopted the convention $f(0,d)^0=2$). To get \eqref{to_prove}, then, it is enough to verify
\begin{equation} \label{to_prove-regular}
f(d,d)^d \geq \left(\frac{f(d-1,d)}{f(d,d)}\right)^{d(d-1)} + \frac{\prod_{i=1}^k \left(2^d + 2^{x_i} - 1\right)^\frac{1}{d}}{f(d,d)^{kd}}
\end{equation}
for all $d-1 \leq k \leq d(d-1)$ and for all $(x_1,\ldots, x_k) \in \{0, \ldots, d-1\}^k$ with $\sum_{i=1}^k x_i = kd-d(d-1)$.

In the extreme case $k=d-1$, \eqref{to_prove-regular} becomes
$$
\left(2^{d+1}-1\right)^\frac{1}{d} \geq \frac{2^d + 2^{d-1} - 1}{\left(2^{d+1}-1\right)^\frac{d-1}{d}} +  \frac{2^{d-1}}{\left(2^{d+1}-1\right)^\frac{d-1}{d}}
$$
which holds with equality. Since only the third term of \eqref{to_prove-regular} involves $k$, what we have to show is that
$$
\frac{2^{d-1}}{\left(2^{d+1}-1\right)^\frac{d-1}{d}} \geq \frac{\prod_{i=1}^k \left(2^d + 2^{x_i} - 1\right)^\frac{1}{d}}{\left(2^{d+1}-1\right)^{\frac{k}{d}}}
$$
or, equivalently,
\begin{equation} \label{to_prove-regular'}
\left(2^{d+1}-1\right)^{k-(d-1)} 2^{d(d-1)}\geq \prod_{i=1}^k \left(2^d + 2^{x_i} - 1\right)
\end{equation}
for all $d-1 \leq k \leq d(d-1)$ and for all $(x_1,\ldots, x_k) \in \{0, \ldots, d-1\}^k$ with $\sum_{i=1}^k x_i = kd-d(d-1)$ (and, without loss of generality, $x_1 \geq \ldots \geq x_k$), with equality only if all $x_i=0$ (and $k=d-1$).

For a sequence $(x_1, \ldots, x_k)$, set
$g(x_1, \ldots, x_k) = \prod_{i=1}^k \left(2^d + 2^{x_i} - 1\right)$.
An equivalent formulation of \eqref{to_prove-regular'} is that for all sequences $(x_1, \ldots, x_k)$ that are valid for the right-hand side of \eqref{to_prove-regular'},
$$
g(d, \ldots, d, 0, \ldots, 0) \geq g(x_1, \ldots, x_k)
$$
where there are $k-(d-1)$ $d$'s and $d-1$ $0$'s in the string on the left-hand side, with equality only if $k=d-1$ and all $x_i=0$. By a sequence of moves in which two terms $x_i$ and $x_j$ with $x_i \geq x_i$ are replaced by $x_i+1$ and $x_j-1$, it is possible to transform any valid sequence $(x_1, \ldots, x_k)$ into $(d, \ldots, d, 0, \ldots, 0)$. So \eqref{to_prove-regular'} follows from the observation that the function $g(x)=2^d + 2^x -1$ (defined on positive integers) has the property that $\frac{g(x+1)}{g(x)}$ is strictly increasing in $x$ and so $g(x_i+1)g(x_j-1) > g(x_i)g(x_j)$ for $x_i \geq x_j$.
\end{proof}

Armed with Lemmas \ref{lem-obs_1} and \ref{lem-regular}, we are now in a position to complete the proof of Proposition \ref{prop-Delta-at-most-5}. Statement \ref{main-thm_2} yields to a direct search.
\begin{lemma} \label{lem-Delta-4-search'}
Let $x$ be any vertex of maximum degree in $G \in {\cal G}^{\rm bipartite}(4)$. Then $x$ is good for $G$. There is equality in \eqref{to_prove} if and only if $x$ is an isolated vertex, or the component of $x$ is a complete bipartite graph.
\end{lemma}

\begin{proof}
Exhaustive analysis of all possible situations that can arise in \eqref{to_prove}, always taking level $3$ vertices to have degree $4$.
\end{proof}

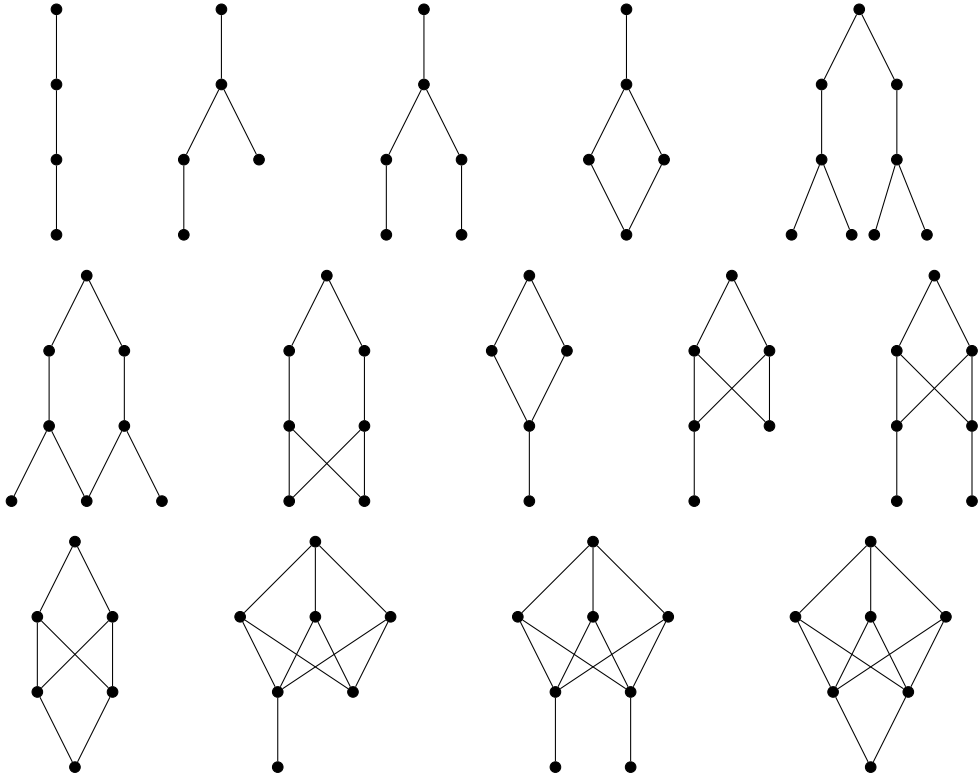
\begin{figure}[ht!] \centering	
	\newcommand{\hsep}{\hspace{1.4cm}}
	\begin{tikzpicture}
		\node[V] (0) at (0,0) {};
		\node[V] (1) at (0,-1) {} edge (0);
		\node[V] (2) at (0,-2) {} edge (1);
		\node[V] (3) at (0,-3) {} edge (2);
	\end{tikzpicture}
	\hsep
	\begin{tikzpicture}
		\node[V] (0) at (0,0) {};
		\node[V] (1) at (0,-1) {} edge (0);
		\node[V] (2) at (-.5,-2) {} edge (1);
		\node[V] (3) at (.5,-2) {} edge (1);
		\node[V] (4) at (-.5,-3) {} edge (2);
	\end{tikzpicture}
	\hsep
	\begin{tikzpicture}
		\node[V] (0) at (0,0) {};
		\node[V] (1) at (0,-1) {} edge (0);
		\node[V] (2) at (-.5,-2) {} edge (1);
		\node[V] (3) at (.5,-2) {} edge (1);
		\node[V] (4) at (-.5,-3) {} edge (2);
		\node[V] (5) at (.5,-3) {} edge (3);
	\end{tikzpicture}
	\hsep
	\begin{tikzpicture}
		\node[V] (0) at (0,0) {};
		\node[V] (1) at (0,-1) {} edge (0);
		\node[V] (2) at (-.5,-2) {} edge (1);
		\node[V] (3) at (.5,-2) {} edge (1);
		\node[V] (4) at (0,-3) {} edge (2) edge (3);
	\end{tikzpicture}
	\hsep
	\begin{tikzpicture}
		\node[V] (0) at (0,0) {};
		\node[V] (1) at (-.5,-1) {} edge (0);
		\node[V] (2) at (.5,-1) {} edge (0);
		\node[V] (3) at (-.5,-2) {} edge (1);
		\node[V] (4) at (.5,-2) {} edge (2);
		\node[V] (5) at (-0.9,-3) {} edge (3);
		\node[V] (6) at (-.1,-3) {} edge (3);
		\node[V] (7) at (.2,-3) {} edge (4);
		\node[V] (8) at (0.9,-3) {} edge (4);
	\end{tikzpicture}

	\vspace{10pt}
	\begin{tikzpicture}
		\node[V] (0) at (0,0) {};
		\node[V] (1) at (-.5,-1) {} edge (0);
		\node[V] (2) at (.5,-1) {} edge (0);
		\node[V] (3) at (-.5,-2) {} edge (1);
		\node[V] (4) at (.5,-2) {} edge (2);
		\node[V] (5) at (-1,-3) {} edge (3);
		\node[V] (6) at (-0,-3) {} edge (3) edge (4);
		\node[V] (7) at (1,-3) {} edge (4);
	\end{tikzpicture}
	\hsep
    \begin{tikzpicture}
		\node[V] (0) at (0,0) {};
		\node[V] (1) at (-.5,-1) {} edge (0);
		\node[V] (2) at (.5,-1) {} edge (0);
		\node[V] (3) at (-.5,-2) {} edge (1);
		\node[V] (4) at (.5,-2) {} edge (2);
		\node[V] (5) at (-.5,-3) {} edge (3) edge (4);
		\node[V] (6) at (.5,-3) {} edge (3) edge (4);
	\end{tikzpicture}
	\hsep
	\begin{tikzpicture}
		\node[V] (0) at (0,0) {};
		\node[V] (1) at (-.5,-1) {} edge (0);
		\node[V] (2) at (.5,-1) {} edge (0);
		\node[V] (3) at (0,-2) {} edge (1) edge (2);
		\node[V] (4) at (0,-3) {} edge (3);
	\end{tikzpicture}
    \hsep
	\begin{tikzpicture}
		\node[V] (0) at (0,0) {};
		\node[V] (1) at (-.5,-1) {} edge (0);
		\node[V] (2) at (.5,-1) {} edge (0);
		\node[V] (3) at (-.5,-2) {} edge (1) edge (2);
		\node[V] (4) at (.5,-2) {} edge (1) edge (2);
		\node[V] (5) at (-.5,-3) {} edge (3);
	\end{tikzpicture}
	\hsep
	\begin{tikzpicture}
		\node[V] (0) at (0,0) {};
		\node[V] (1) at (-.5,-1) {} edge (0);
		\node[V] (2) at (.5,-1) {} edge (0);
		\node[V] (3) at (-.5,-2) {} edge (1) edge (2);
		\node[V] (4) at (.5,-2) {} edge (1) edge (2);
		\node[V] (5) at (-.5,-3) {} edge (3);
		\node[V] (5) at (.5,-3) {} edge (4);
	\end{tikzpicture}

	\vspace{10pt}
	\begin{tikzpicture}
		\node[V] (0) at (0,0) {};
		\node[V] (1) at (-.5,-1) {} edge (0);
		\node[V] (2) at (.5,-1) {} edge (0);
		\node[V] (3) at (-.5,-2) {} edge (1) edge (2);
		\node[V] (4) at (.5,-2) {} edge (1) edge (2);
		\node[V] (5) at (0,-3) {} edge (3) edge (4);
	\end{tikzpicture}
	\hsep
	\begin{tikzpicture}
		\node[V] (0) at (0,0) {};
		\node[V] (1) at (-1,-1) {} edge (0);
		\node[V] (2) at (0,-1) {} edge (0);
		\node[V] (3) at (1,-1) {} edge (0);
		\node[V] (4) at (-.5,-2) {} edge (1) edge (2) edge (3);
		\node[V] (5) at (.5,-2) {} edge (1) edge (2) edge (3);
		\node[V] (6) at (-.5,-3) {} edge (4);
	\end{tikzpicture}
	\hsep
	\begin{tikzpicture}
		\node[V] (0) at (0,0) {};
		\node[V] (1) at (-1,-1) {} edge (0);
		\node[V] (2) at (0,-1) {} edge (0);
		\node[V] (3) at (1,-1) {} edge (0);
		\node[V] (4) at (-.5,-2) {} edge (1) edge (2) edge (3);
		\node[V] (5) at (.5,-2) {} edge (1) edge (2) edge (3);
		\node[V] (6) at (-.5,-3) {} edge (4);
		\node[V] (7) at (.5,-3) {} edge (5);
	\end{tikzpicture}
	\hsep
	\begin{tikzpicture} 	
		\node[V] (0) at (0,0) {};
		\node[V] (1) at (-1,-1) {} edge (0);
		\node[V] (2) at (0,-1) {} edge (0);
		\node[V] (3) at (1,-1) {} edge (0);
		\node[V] (4) at (-.5,-2) {} edge (1) edge (2) edge (3);
		\node[V] (5) at (.5,-2) {} edge (1) edge (2) edge (3);
		\node[V] (6) at (0,-3) {} edge (4) edge (5);
	\end{tikzpicture}
	\caption{The levels 0,1,2 and 3 appearances of the exceptional graphs in Lemma \ref{lem-Delta-5-search}. The graphs are drawn with $x$ as the top vertex.\label{fig:exception}}
\end{figure}

Now we consider the case when the maximum degree is at most $5$. We are not able to verify the analogue of Lemma \ref{lem-Delta-4-search'}, as there are too many cases to examine. To reduce the number of cases, we take as a working hypothesis that any minimum degree vertex in $G \in {\cal G}^{\rm bipartite}(5)$ is good. This turns out to be false, but only just.

\begin{lemma} \label{lem-Delta-5-search}
Let $G$ be a graph in ${\cal G}^{\rm bipartite}(5)$ that is not regular. Let $x \in V(G)$ be a vertex of minimum degree. Either $x$ is good for $G$ or the subgraph of $G$ induced by $x$ and all the level $1$, $2$ and $3$ vertices has one of the fourteen appearances shown in Figure \ref{fig:exception}.

There is equality in \eqref{to_prove} if and only if there are no level $3$ vertices, and the level $0$, $1$ and $2$ vertices together induce either a single vertex or a complete bipartite graph.
\end{lemma}

\begin{proof}
Again, exhaustive case analysis.
\end{proof}

To complete the proof of Proposition \ref{prop-Delta-at-most-5} we need to show that the exceptional graphs of Lemma \ref{lem-Delta-5-search} (the fourteen graphs shown in Figure \ref{fig:exception}) all contain good vertices. To accomplish this, we take $x'$ to be any neighbour of $x$, and examine \eqref{to_prove} for every instance in which the level $0$, $1$ and $2$ vertices (measured relative to $x'$) induce a subgraph consistent with one of the exceptional graphs. There are only finitely many such cases; an exhaustive search verifies that in each case $x'$ is good for $x$ (with no case leading to equality in \eqref{to_prove}).

\medskip

The program used to conduct the searches required for Lemmas \ref{lem-Delta-4-search'} and \ref{lem-Delta-5-search} is available, with documentation, at \url{http://www.nd.edu/~dgalvin1/research.html#graphscombinatorics}.

\bigskip

\noindent
{\bf Acknowledgements.} The first author is supported by National Security Agency grant H98230-10-1-0364. The second author performed this research at the University of Minnesota Duluth under the supervision of Joseph Gallian with the support of National Science Foundation and Department of Defense grant DMS 0754106, National Security Agency grant H98230-06-1-0013, and the MIT Department of Mathematics.


\begin{thebibliography}{99}

\bibitem{Alon}
N. Alon, Independent sets in regular graphs and sum-free subsets
of finite groups, {\em Israel J. Math.} {\bf 73} (1991), 247--256.

\bibitem{Kahn}
J. Kahn, An entropy approach to the hard-core model on bipartite
graphs, {\em Combin. Probab. Comput.} {\bf 10} (2001), 219--237.

\bibitem{Zhao}
Y. Zhao, The Number of Independent Sets in a Regular Graph,
{\em Combin. Probab. Comput.} {\bf 19} (2010), 315--320.

\end{thebibliography}
\end{document}